\documentclass[a4paper,reqno]{amsart}
\usepackage[english]{babel}
\usepackage{amssymb, amsmath, amsthm,mathrsfs} 
\usepackage{bbm,amsmath,amsthm,epsfig,latexsym,marvosym, esint}
\usepackage{amsfonts}
\usepackage{bigints}
\usepackage{amsmath}
\usepackage{bbm}
\usepackage{hyperref}

\usepackage{fancyhdr}
\usepackage{enumitem}
\usepackage{amssymb}
\usepackage{ esint }
\usepackage{hyperref}
\usepackage{color}
\usepackage[font={small,up}]{caption}

\numberwithin{equation}{section} 

\usepackage{color}
\usepackage{hyperref}
\definecolor{ddorange}{rgb}{1,0.5,0}
\definecolor{ddcyan}{rgb}{0,0.2,1.0}

\def\Xint#1{\mathchoice
{\XXint\displaystyle\textstyle{#1}}%
{\XXint\textstyle\scriptstyle{#1}}%
{\XXint\scriptstyle\scriptscriptstyle{#1}}%
{\XXint\scriptscriptstyle\scriptscriptstyle{#1}}%
\!\int}
\def\XXint#1#2#3{{\setbox0=\hbox{$#1{#2#3}{\int}$}
\vcenter{\hbox{$#2#3$}}\kern-.5\wd0}}

\def\dashint{\Xint-}

\usepackage{amsthm}
\makeatletter
\def\th@plain{%
  \thm@notefont{}
  \itshape 
}
\def\th@definition{%
  \thm@notefont{}
  \normalfont 
}
\makeatother
\newtheorem{thm}{Theorem}[section]
\newtheorem{prop}[thm]{Proposition}
\newtheorem{lem}[thm]{Lemma}

\theoremstyle{definition}
\newtheorem{defn}[thm]{Definition}

\newtheorem{oss}[thm]{Remark}

\newcommand{\N}{\mathbb{N}}
\newcommand{\R}{\mathbb{R}}

\renewcommand{\epsilon}{\varepsilon}

\newcommand{\sm}{\setminus}

\newcommand{\Rd}{\mathbb{R}^n}

\newcommand{\Mdd}{{\R^{n\times n}_{sym}}}

\newcommand{\e}{\varepsilon}

\title[On some non-local approximation of nonisotropic Griffith-type functionals]{On some non-local approximation of nonisotropic Griffith-type functionals}
\author[Fernando Farroni]{Fernando Farroni}
\address[Fernando Farroni]{Dipartimento di Matematica ed Applicazioni ``R. Caccioppoli'', Universit\`{a} di Napoli Federico~II, Via Cintia Monte Sant'Angelo, 80126 Napoli, Italy}
\email[Fernando Farroni]{fernando.farroni@unina.it}
\author[Giovanni Scilla]{Giovanni Scilla}
\address[Giovanni Scilla]{Dipartimento di Scienze di Base e Applicate per l'Ingegneria (SBAI), Sapienza Universit\`{a} di Roma, Via A. Scarpa 16, 00161 Roma, Italy}
\email[Giovanni Scilla]{giovanni.scilla@uniroma1.it}
\author[Francesco Solombrino]{Francesco Solombrino}
\address[Francesco Solombrino]{Dipartimento di Matematica ed Applicazioni ``R. Caccioppoli'', Universit\`{a} di Napoli Federico~II, Via Cintia Monte Sant'Angelo, 80126 Napoli, Italy}
\email[Francesco Solombrino]{francesco.solombrino@unina.it}

\begin{document}
\begin{abstract}
The approximation in the sense of $\Gamma$\hbox{-}convergence of nonisotropic Griffith-type functionals, with $p-$growth ($p>1$) in the symmetrized gradient, by means of a suitable sequence of non-local convolution type functionals defined on Sobolev spaces, is analysed.
\end{abstract}
\keywords{non-local approximations, $\Gamma$\hbox{-}convergence, Griffith functional, brittle fracture}
\subjclass{49Q20; 49J45; 74R10}

\maketitle

\setcounter{tocdepth}{1}  
\tableofcontents


\bigskip
\bigskip

\section{Introduction}
The scope of this paper is to provide a generalization of recent results, obtained in \cite{SS21}, concerning the approximation of brittle fracture energies for linearly elastic materials, by means of {\it nonlocal functionals} defined on Sobolev spaces, which are easier to handle also from a computational point of view.

In \cite{SS21} an approach originally devised by Braides and Dal Maso~\cite{BraDal97} 
for the approximation of the Mumford-Shah functional has been generalized to the linearly elastic context. Namely, it was shown that, for a given bounded increasing function $f\colon \R^+\to \R^+$  the energies
\begin{equation*}
F_\epsilon(u):=\frac1\varepsilon\int_\Omega f\left(\epsilon\dashint_{B_\epsilon(x)\cap\Omega} W\left(\mathcal Eu(y)\right)\,\mathrm{d}y\right)\,\mathrm{d}x
\end{equation*}
$\Gamma$-converge to the functional 
\[
\alpha\int_{\Omega} W(\mathcal{E}u(x))\,\mathrm{d}x  + 2\beta \mathcal{H}^{d-1}(J_u)\,,
\]
with $\alpha=f^\prime(0)$ and $\beta=\lim_{t\to +\infty}f(t)$, in the $L^1(\Omega)$-topology. Above, $W(\mathcal Eu(y))$ is a convex elastic energy depending on the linearized strain $\mathcal Eu$, given by the symmetrized gradient of a vector-valued displacement $u$, whose jump set $J_u$ represents the cracked part of a material. The energy space of the limit functional is the one of generalized functions with bounded deformation, introduced in \cite{DM2013}. 

It is noteworthy that the above result allowed one for a general (convex) bulk energy $W$ having $p$-growth for $p>1$. The proof strategy must then avoid, at least when estimating the bulk part, any slicing procedure. This latter is instead successful in the special case \footnote{we  remark that this particular case is however not the most relevant one from a mechanical point of view, as even for an isotropic material additional terms in the bulk energy are expected to appear.}  $W(\xi)=|\xi|^p$ , considered for instance in \cite{Negri2006}. There, non-local convolution-type energies of the form
\begin{equation}
\frac1\varepsilon\int_\Omega f\left(\varepsilon\int_{\R^n} \left|\mathcal Eu(y)\right|^p\rho_\epsilon(x-y)\,\mathrm{d}y\right)\,\mathrm{d}x
\label{eq:Negri}
\end{equation}
are considered, where  $\rho$ is a convolution kernel whose support is a convex bounded domain and $\rho_\epsilon(z)$ is the usual sequence
of convolution kernels $\rho(z/\epsilon)/\epsilon^d$. The $\Gamma$\hbox{-}limit of \eqref{eq:Negri} with respect to the $L^1$ convergence is given by
\begin{equation*}
\int_{\Omega}|\mathcal{E}u(x)|^p\,\mathrm{d}x  + \int_{J_u} \phi_\rho(\nu)\,\mathrm{d}\mathcal{H}^{d-1}\,, 
\end{equation*}
where the anisotropic surface density $\phi_\rho$ depends on the geometry and on the size of ${\rm supp }\rho$.  A similar effort of generalizing the results of \cite{BraDal97} to Mumford-Shah type energies with non-isotropic surface part has been previously performed in \cite{CT}.

In this paper, we extend the focus of \cite{Negri2006, SS21} by showing that general Griffith-type functionals of the form
\begin{equation}\label{eq:Gr}
\alpha\int_{\Omega} W(\mathcal{E}u(x))\,\mathrm{d}x  + 2\beta \int_{J_u} \phi(\nu)\,\mathrm{d}\mathcal{H}^{d-1}\,,
\end{equation}
where $\phi$ is any norm on $\R^n$, can be obtained as variational limit of non-local convolution-type functionals 
\[
\frac1\varepsilon\int_\Omega f\left(\varepsilon\int_{\R^n}W(\mathcal Eu(y))\rho_\epsilon(x-y)\,\mathrm{d}y\right)\,\mathrm{d}x\,.
\]
Above,  $f$ is again a bounded nondecreasing function with $\alpha=f^\prime(0)$ and $\beta=\lim_{t\to +\infty}f(t)$, and the unscaled kernel $\rho$ has the bounded convex symmetric domain $\overline{S}:=~\{\xi\in \R^n \colon \phi(\xi)\le 1\}$ as its support. This is the analogue, in the linear elastic setting, of the results in \cite{CT}.

The proof strategy we devise is based on a localization method and involves nontrivial adaptions to the method used in \cite{SS21}, in particular when estimating the bulk term in the $\Gamma$-liminf inequality (Proposition \ref{prop:estimate}). There, we have to impose (and this is the only point in the paper) an additional restriction on the convolution kernel $\rho$, namely of being nonincreasing with respect to the given norm $\phi$ (see Assumption \ref{ass-N2} below). This is namely needed in order to be able to estimate from below the size of the nonlocal approximations of the bulk term in an anisotropic tubular neighborhood of the set where they exceed the threshold $\frac{\beta}\alpha$, which heuristically corresponds to the breaking of the elastic bonds. With this, a set $K^\prime_\varepsilon$ with small area and bounded perimeter, where the fracture energy concentrates can be explicitly constructed. This yields an estimate of the $\Gamma$-liminf which has an optimal constant in front of the bulk term, although being non-optimal for the surface energy.

Another non-optimal estimate for the $\Gamma$-liminf, but with an optimal constant for the surface energy can be instead obtained by a slicing procedure, involving a comparison argument and the convexity of the open set $S$ (Proposition \ref{prop:lowboundjump}). As bulk and surface energy in \eqref{eq:Gr} are mutually singular as measures, a localization procedure entail then  the $\Gamma$-liminf inequality (Proposition \ref{prop:lowerbound}).  Finally, the $\Gamma$-limsup inequality (Proposition \ref{prop:upperbound}) can be obtained by a direct construction for a regular class of competitors having a ``nice'' jump set, and which are dense in energy. Notice indeed that such an approximation (see Theorem \ref{thm: approx} for a precise statement) is possibile also with respect to an anisotropic norm $\phi$, combining  the recent results in \cite{CC} with the ones in \cite{CorToa99}. 

As a final remark, it would be desirable to get rid on the structural assumption \ref{ass-N2} on the convolution kernels, which is used only in Proposition \ref{prop:estimate}.  It is our opinion that this is going to require quite a delicate abstract analysis of the $\Gamma$-limit of nonlocal functionals which approximate free-discontinuity problems in $GSBD$, possibly including also finite-difference models which are well suited to numerical approximations (see \cite{CSS} for a recent discrete finite-difference approximation of some Griffith-type functionals in $GSBD$). A similar analysis for the $SBV$ setting has been performed in \cite{Cortesani1998357}, where integral representation formulas for the limit energy have been provided. Furthermore, nontrivial sufficient conditions have been given under which the bulk part of the energy can be recovered by only considering weakly compact sequences in Sobolev spaces.  We plan to defer this abstract analysis to a forthcoming contribution. For the asymptotic analysis via $\Gamma$-convergence of local free-discontinuity functionals in linear elasticity and the related issues, we refer the reader to the very recent papers \cite{CCS, CriFriSol, homogen}.

\emph{Outline of the paper:}  The paper is structured as follows. In Section~\ref{sec:notation} we fix the basic notation and results on the function spaces we will deal with (Section~\ref{sec: gsbd}), together with some technical lemmas (Section~\ref{sec:lemmas}) which will be useful throughout the paper. In Section~\ref{sec:model} we list the main assumptions, introduce our model (eq. \eqref{energies0}), and state the main results of the paper, provided {in} Theorem~\ref{thm:mainresult} and Theorem~\ref{thm:mainresult2}. Section~\ref{sec:compactnessestimbelow} is devoted to the proof of the compactness statements in the main Theorems (Proposition \ref{prop:compactness}), and to the $\Gamma$\hbox{-}liminf inequality, which is proved in Section~\ref{sec:gammaliminf} combining the estimates in Sections~\ref{sec:estimbelowbulk} and \ref{sec:estimbelowsurf}. The proof of the  upper bound is given in Section~\ref{sec:upperbound}.

\section{Notation and preliminary results} \label{sec:notation}

\subsection{Notation}\label{sec:notation}

The symbol $|\cdot|$ denotes the Euclidean norm in any dimension, while $\langle\cdot,\cdot\rangle$ stands for the scalar product in $\mathbb{R}^n$. We will always denote by $\Omega$ an open, bounded subset of $\mathbb{R}^n$ {with Lipschitz boundary}, and by $\mathbb{S}^{n-1}$ the $(n-1)$-dimensional unit sphere. The Lebesgue measure in $\mathbb{R}^n$ and the $s$-dimensional Hausdorff measure are written as $\mathcal{L}^n$ and $\mathcal{H}^s$, respectively. $\mathcal{A}(\Omega)$ stands for the family of the open subsets of $\Omega$. 

Let $S$ be a bounded, open, convex and \emph{symmetrical} set, i.e. $S=-S$. For $\eta>0$, we denote by $\eta S$ the $\eta$-dilation of $S$ and we will often use the shorthand {$S(x,\eta)$} in place of $x+\eta S$. We consider $|\cdot|_S$ the norm induced by $S$, defined as
\begin{equation}
|x|_S:=\inf\{\eta>0:\,\, x\in\eta S\}\,,
\end{equation}
whose unit ball $\{|x|_S<1\}$ coincides with $S$, and, correspondingly, we introduce the distance to a closed bounded set $K\subset\R^n$; namely, 
\begin{equation}
{\rm dist}_S(x,K):=\min_{y\in K}|x-y|_S\,,\quad x\in\R^n\,.
\end{equation}

\subsection{$GBD$ and $GSBD$ functions}\label{sec: gsbd}

In this section we recall some basic definitions and results on generalized functions with bounded deformation, as introduced in \cite{DM2013}. Throughout the paper we will use standard notations for the spaces $(G)SBV$ and $(G)SBD$, referring the reader to \cite{AFP} and \cite{ACDM, BCDM, Temam}, respectively, for a detailed treatment on the topics.\\

Let $\xi\in\R^n\backslash\{0\}$ and $\Pi^\xi=\{y\in\R^n:\, \langle\xi,y\rangle=0\}$. If $\Omega\subset\R^n$ and $y\in\Pi^\xi$ we set $\Omega_{\xi,y}:=\{t\in\R:\, y+t\xi\in \Omega\}$ and $\Omega_\xi:=\{y\in \Pi^\xi:\, \Omega_{\xi,y}\neq\emptyset\}$. Given $u:\Omega\to\R^n$, $n\geq2$, we define $u^{\xi,y}: \Omega_{\xi,y}\to\R$ by 
\begin{equation}
u^{\xi,y}(t):=\langle u(y+t\xi),\xi\rangle\,, 
\label{section1}
\end{equation}
while if $v: \Omega\to\R$, the symbol $v^{\xi,y}$ will denote the restriction of $v$ to the set $\Omega_{\xi,y}$; namely,
\begin{equation}
v^{\xi,y}(t):= v(y+t\xi)\,.
\label{section2}
\end{equation}

Let $\xi\in \mathbb{S}^{n-1}$. For any $x\in\R^n$ we denote by $x_\xi$ and $y_\xi$ the projections onto the subspaces $\Xi:=\{t\xi:\,\,t\in\R\}$ and $\Pi^\xi$, respectively. For $\sigma, r >0$ and $x\in\R^n$ we define the cylinders
\begin{equation*}
C_{\sigma,r}^\xi(0):=\{x\in\R^n:\,\, |x_\xi|<\sigma\,,\,\, |y_\xi|<r\}\,,\quad C_{\sigma,r}^\xi(x):=x+C_{\sigma,r}^\xi(0)\,.
\end{equation*}
Note that $C_{\sigma,r}^\xi(x)=(x_\xi-\sigma,x_\xi+\sigma)\times B^{n-1}_{r}(y_\xi)$, where $B^{n-1}$ denotes a ball in the $(n-1)$-dimensional space $\Pi^\xi$. 

\begin{defn}
An $\mathcal L^{n}$-measurable function $u:\Omega\to \R^{n}$ belongs to $GBD(\Omega)$ if there exists a positive bounded Radon measure $\lambda_u$ such that, for all $\tau \in C^{1}(\R^{n})$ with $-\frac12 \le \tau \le \frac12$ and $0\le \tau'\le 1$, and all $\xi \in \mathbb{S}^{n-1}$, the distributional derivative $D_\xi (\tau(\langle u,\xi\rangle))$ is a bounded Radon measure on $\Omega$ whose total variation satisfies
$$
\left|D_\xi (\tau(\langle u,\xi\rangle))\right|(B)\le \lambda_u(B)
$$
for every Borel subset $B$ of $\Omega$. 
\end{defn}

If $u\in GBD(\Omega)$ and $\xi\in\R^n\backslash\{0\}$ then, in view of \cite[Theorem~9.1, Theorem~8.1]{DM2013}, the following properties hold:
\begin{enumerate}
\item[{\rm(a)}] $\dot{u}^{\xi,y}(t)=\langle\mathcal{E}u(y+t\xi)\xi,\xi\rangle$ for a.e. $t\in \Omega_{\xi,y}$;\\
\item[{\rm(b)}] {$J_{u^{\xi,y}}=(J_u^\xi)_{\xi,y}$} for $\mathcal{H}^{n-1}$-a.e. $y\in\Pi^\xi$, where
\begin{equation*}
J_u^\xi:=\{x\in J_u:\, \langle u^+(x)-u^-(x),\xi\rangle\neq0\}\,.
\label{jumpset}
\end{equation*}
\end{enumerate}

\begin{defn}
A function $u \in GBD(\Omega)$ belongs to the subset $GSBD(\Omega)$ of special functions of bounded deformation if, in addition, for every $\xi \in \mathbb{S}^{n-1}$ and $\mathcal H^{n-1}$-a.e.\ $y \in \Pi^\xi$, it holds that $u^{\xi,y}\in SBV_{\mathrm{loc}}(\Omega_{\xi,y})$.
\end{defn}

The inclusions $BD(\Omega)\subset GBD(\Omega)$ and $SBD(\Omega)\subset GSBD(\Omega)$ hold (see \cite[Remark 4.5]{DM2013}).  {Although} they are, in general, strict,  relevant properties of $BD$ functions are retained also in this weak setting. In particular, $GBD$-functions have an approximate symmetric differential $\mathcal{E}u(x)$ at $\mathcal L^{n}$-a.e.\ $x\in \Omega$. Furthermore the jump set $J_u$ of a $GBD$-function is $\mathcal H^{n-1}$-rectifiable (this is proven in \cite[Theorem 6.2 and Theorem 9.1]{DM2013}, but it has been recently shown that this property is actually a general one for measurable functions \cite{DN}).

{Let $p>1$.} The space $GSBD^p(\Omega)$ is defined as
$$
GSBD^p (\Omega):= \{u \in GSBD(\Omega): \mathcal{E}u \in L^p (\Omega; \mathbb R_{\mathrm{sym}}^{n\times n})\,,\,\mathcal H^{n-1}(J_u) < +\infty\}\,.
$$

Every function in $GSBD^p(\Omega)$ can be approximated with the so-called ``piecewise smooth''  $SBV$-functions, denoted $\mathcal{W}(\Omega;\R^n)$, characterized by the three properties 
\begin{equation}\label{1412191008}
\begin{cases}
u\in SBV(\Omega;\R^n)\cap W^{m,\infty}(\Omega\sm J_u;\R^n) \,\text{for every }m\in \N\,,\\
\mathcal{H}^{n-1}(\overline{J}_u \sm J_u ) = 0\,,\\
\overline{J}_u \text{ is the intersection of $\Omega$ with a finite union of ${(n{-}1)}$-dimensional simplexes}\,. 
\end{cases}
\end{equation}
This is stated by the following result, which combines \cite[Theorem~1.1]{CC} with  \cite[Theorem~3.9]{CorToa99}. 

\begin{thm}\label{thm: approx}
Let $\phi$ be a norm on $\R^n$. Let $\Omega\subset\mathbb{R}^n$ be a bounded open Lipschitz set, and let $u\in GSBD^p(\Omega;\R^n)$. Then there exists a sequence {$(u_j)$} such that $u_j\in \mathcal{W}(\Omega;\Rd)$ and
\begin{align}
u_j\to u \mbox{ in measure on $\Omega$},\label{convme}\\
\mathcal{E}u_j \to \mathcal{E}u \mbox{ in $L^p(\Omega;\R^{n\times n}_{sym})$,}\label{convgrad}\\
\int_{J_{u_j}}\phi(\nu_{u_j})\mathcal{H}^{n-1}\to \int_{J_{u}}\phi(\nu_u)\mathcal{H}^{n-1}\,.\label{convjump}
\end{align}
Moreover, if $\int_\Omega \psi(|u|)\,\mathrm{d}x$ is finite for $\psi:[0,+\infty)\to[0,+\infty)$ continuous, {increasing}, with
\begin{equation*}
\psi(0)=0,\,\,\, \psi(s+t)\le C(\psi(s)+\psi(t)), \,\,\, \psi(s)\le C(1+s^p), \,\,\, {\lim_{s\to+\infty}{\psi(s)}=+\infty}
\end{equation*} 
then
\begin{equation}
\lim_{j\to+\infty}\int_\Omega \psi(|u_j-u|)\,\mathrm{d}x=0\,.
\label{eq:densityfidel}
\end{equation}
\label{thm:density}
\end{thm}

As observed in  \cite[Remark~4.3]{ChaCri18b}, we may even approximate through functions $u$ such that,  besides \eqref{1412191008}, have a {\it closed} jump set strictly contained in $\Omega$ made of {\it pairwise disjoint} $(n{-}1)$-dimensional simplexes, with $J_u \cap \Pi_i \cap \Pi_l=\emptyset$ for any two different hyperplanes $\Pi_i$, $\Pi_l$.

We recall the following general $GSBD^p$ compactness result from \cite{CC18Comp}, which generalizes \cite[Theorem 11.3]{DM2013}. 
\begin{thm}[$GSBD^p$ compactness]\label{th: GSDBcompactness}
 Let $\Omega \subset \R^n$ be an open, bounded set,  and let $(u_j)_j \subset  GSBD^p(\Omega)$ be a sequence satisfying
$$ \sup\nolimits_{j\in \N} \big( \Vert \mathcal{E}u_j \Vert_{L^p(\Omega)} + \mathcal{H}^{n-1}(J_{u_j})\big) < + \infty.$$
Then there exists a subsequence, still denoted by {$(u_j)$}, such that the set  ${A^\infty} := \lbrace x\in \Omega: \, |u_j(x)| \to +\infty \rbrace$ has finite perimeter, and  there exists  $u \in GSBD^p(\Omega)$ such that 
\begin{align}\label{eq: GSBD comp}
{\rm (i)} & \ \ u_j \to u \  \ \ \  \mbox{ in measure on $\Omega\sm A^\infty$}, \notag \\ 
{\rm (ii)} & \ \ \mathcal{E}u_j \rightharpoonup\mathcal{E}u \ \ \ \text{ in } L^p(\Omega \setminus A^\infty; \Mdd),\notag \\
{\rm (iii)} & \ \ \liminf_{j \to \infty} \mathcal{H}^{n-1}(J_{u_j}) \ge \mathcal{H}^{n-1}(J_u \cup  (\partial^*A^\infty \cap\Omega)  )\,,
\end{align}
where $\partial^*$ denotes the essential boundary of a set with finite perimeter.
\end{thm}

\begin{oss}\label{rem: compactness}
If in the statement above one additionally assumes that
\[
 \displaystyle\sup_{j\in \N}\int_\Omega  \psi(|u_j|)\,\mathrm{d}x<+\infty
\]
for a positive, continuous and increasing function $\psi$ with $\lim_{s\to +\infty}\psi(s)=+\infty$, then {$A^\infty=\emptyset$}, so that $|u|$ is finite a.e., and (i) holds on $\Omega$. Moreover, if $\psi$ is superlinear at infinity, that is
\begin{equation*}
\lim_{s\to+\infty}\frac{\psi(s)}{s}=+\infty\,,
\end{equation*}
by the Vitali dominated convergence theorem one gets that  $u \in L^1(\Omega)$ and ${\rm (i)}$ holds with respect to the $L^1$-convergence in $\Omega$.
\end{oss}

\subsection{Some lemmas}\label{sec:lemmas}

We recall here (without adding the standard proofs) some properties of integral convolutions in the setting of Sobolev spaces. 

\begin{prop}\label{prop:convolutions}
Let $w\in W^{1,p}(\Omega;\R^n)$ and $\rho\in L^\infty(\R^n)$ be a convolution kernel, with ${\rm supp}\,\rho\subset\overline{S}$ for some bounded, open and convex set $S\subset\Omega$.  Set $\rho_\theta(x):=\frac1{\theta^d}\rho\left(\frac x\theta\right)$. Then the following holds:
\begin{itemize}
\item[(i)] let $\Omega' \subset\subset \Omega$ and $0\le \theta \le \mathrm{dist}_S(\Omega', \partial \Omega)$. The convolution
\begin{equation*}
\varphi^\theta(x):=\int_{\Omega} w(y)\rho_\theta(y-x)\,\mathrm{d}y
\end{equation*}
belongs to $W^{1,p}(\Omega';\R^n)$. Moreover, it holds that
\begin{equation}
\nabla \varphi^\theta (x) = \int_{\Omega} \nabla w(y)\rho_\theta(y-x)\,\mathrm{d}y \quad \mbox{ a.e. on $\Omega'$.}
\label{commut}
\end{equation}
\item[(ii)] assume that $w_{\e}\rightarrow w$ in $L^1(\Omega;\R^n)$ and let $\theta_\varepsilon$ be any sequence with $\theta_\varepsilon \to 0$ when $\varepsilon \to 0$. Then the sequence
	\[
	\hat{w}_\e(x):=\int_{\Omega} w_{\e}(y)\rho_{\theta_\varepsilon}(y-x)\,\mathrm{d} y 
	\]
satisfies $\hat{w}_\e \to c w$ in $L^1(\Omega;\R^n)$, where $c=\int_{\R^n} \rho(x)\,\mathrm{d}x$.
\end{itemize}
\end{prop}

We also recall the following convergence property of one-dimensional sections of averaged functions (see, e.g., \cite[Lemma~2.7(ii)]{SS21}). 

\begin{lem}\label{lem:lbp}
Assume that $w_{\e}\rightarrow w$ in $L^1(\Omega;\R^n)$ and let $\eta_\varepsilon$ be any sequence with $\eta_\varepsilon \to 0$ when $\varepsilon \to 0$. Then for all $\xi \in \mathbb{S}^{n-1}$ and a.e. \ $y \in \Pi^{\xi}$, the sequence
\[
\hat{w}^{\xi, y}_\e(t):=\dashint_{B^{n-1}_{\eta_\varepsilon}(y)} w_{\e}(z+t\xi)\,\mathrm{d}z
\]
satisfies  $\hat{w}^{\xi, y}_\e \to w^{\xi, y}$ in $L^1(\Omega_{\xi, y} ;\R^n)$, where $w^{\xi, y}(t):=w(y+t\xi)$.
\end{lem}

We will also make use of the following localization result, dealing with the supremum of a family of measures (see, e.g., \cite[Proposition 1.16]{Braides1}).

\begin{lem}\label{lem: lemmasup}
Let $\mu:\mathcal{A}(\Omega)\longrightarrow[0,+\infty)$ be a superadditive function on disjoint open sets, let $\lambda$ be a positive measure on $\Omega$ and let $\varphi_h: \Omega\longrightarrow[0,+\infty]$ be a countable family of Borel functions such that $\mu(A)\geq\int_A\varphi_h\,\mathrm{d}\lambda$ for every $A\in\mathcal{A}(\Omega)$. Then, setting $\varphi:=\sup_{h\in\N}\varphi_h$, it holds that
\begin{equation*}
\mu(A)\geq\int_A\varphi\,\mathrm{d}\lambda
\end{equation*}
for every $A\in\mathcal{A}(\Omega)$.
\end{lem}

Lower semicontinuous increasing functions can be approximated from below with truncated affine functions. We refer the reader to \cite[Lemma~2.10]{SS21} for a proof of the following result. 
\begin{lem}\label{lem: belowapprox}
Consider a lower semicontinuous increasing function $f:[0,+\infty)\to[0,+\infty)$ such that there exist $\alpha, \beta>0$ with 
\[
\lim_{t\to 0^+}\frac{f(t)}t=\alpha, \quad \lim_{t\to +\infty} f(t)=\beta\,.
\]
Then there exist two positive sequences $(a_i)_{i\in \mathbb{N}}$, $(b_i)_{i\in \mathbb{N}}$ with
\[
\sup_i a_i= \alpha, \quad \sup_i b_i=\beta 
\]
and $\min\{a_i t, b_i\}\le f(t)$ for all $i \in \mathbb{N}$ and $t \in \R$.
\end{lem}

\subsection{{$\Gamma$\hbox{-}convergence }} \label{sec:gammaconv}

{{Let $(X,d)$ be a metric space.} We recall here the definition of $\Gamma$\hbox{-}convergence for families of functionals $F_\varepsilon:X\to[-\infty,+\infty]$ depending on a real parameter $\varepsilon$ (see, e.g. \cite{BraidesG, DM93}). 

{For all $u\in X$, we define the \emph{lower $\Gamma$\hbox{-}limit of $(F_\varepsilon)$} as $\varepsilon\to0^+$ by
\begin{equation}
F'(u):=\inf \left\{\displaystyle\mathop{\lim\inf}_{j\to+\infty}F_{\varepsilon_j}(u_j):\,\, \varepsilon_j\to0^+\,,\,\, u_j\to u\right\}\,,
\label{eq:Gammaliminf}
\end{equation}
and the \emph{upper $\Gamma$\hbox{-}limit of $(F_\varepsilon)$} as $\varepsilon\to0^+$ by
\begin{equation}
F''(u):=\inf \left\{\displaystyle\mathop{\lim\sup}_{j\to+\infty}F_{\varepsilon_j}(u_j):\,\, \varepsilon_j\to0^+\,,\,\, u_j\to u\right\}\,.
\label{eq:limsupchar}
\end{equation}
We then say that $(F_\varepsilon)$ \emph{$\Gamma$\hbox{-}converges to $F:X\to[-\infty,+\infty]$ as $\varepsilon\to0^+$} iff 
\begin{equation*}
F(u)=F'(u)=F''(u)\,,\quad \mbox{for all $u\in X$.}
\end{equation*}}

\subsection{A one-dimensional $\Gamma$\hbox{-}convergence result}

The following one-dimensional $\Gamma$\hbox{-}convergence result 
will be useful in the proof of the lower bound for the surface term. 
In the statement below, functions in $L^1(I)$ with $I \subset \R$ are extended by $0$ outside $I$, so that the functionals $H_\varepsilon$ are well-defined (actually, the result is not affected by the considered extension).

\begin{thm}\label{thm:braides}
Let $p>1$, let $I\subset\R$ be a bounded interval and consider a lower semicontinuous, increasing function {$f:[0,+\infty)\to[0,+\infty)$} complying with 
\[
\lim_{t\to 0^+}\frac{f(t)}t=\alpha, \quad \lim_{t\to +\infty} f(t)=\beta
\]
for some $\alpha,\beta>0$. Let $H_\varepsilon:L^1(I)\to[0,+\infty]$ be defined by
\begin{equation*}
H_\varepsilon(u):=\frac{1}{\varepsilon}\int_I f\left(\frac{1}{2}\int_{x-\varepsilon}^{x+\varepsilon}|u'(y)|^p\,\mathrm{d}y\right)\,\mathrm{d}x\,,
\end{equation*}
where it is understood that
\begin{equation*}
 f\left(\frac{1}{2}\int_{x-\varepsilon}^{x+\varepsilon}|u'(y)|^p\,\mathrm{d}y\right)=\beta
\end{equation*}
if $u\not\in W^{1,p}(x-\varepsilon,x+\varepsilon)$. Then the {functionals} $(H_\varepsilon)$ $\Gamma$\hbox{-}converge as $\varepsilon\to0^+$ to the functional
\begin{equation*}
H(u):=
\begin{cases}
\displaystyle \alpha\int_I|u'|^p\,\mathrm{d}t + 2\beta \#(J_u)\,, & \mbox{ if }u\in SBV(I)\,,\\
+\infty\,, & \mbox{ otherwise }
\end{cases}
\end{equation*}
in $L^1(I)$.
\end{thm}

\begin{proof}
The proof can be found, e.g., in \cite[Theorem~3.30]{Braides1}.
\end{proof}


\section{The non-local model and main results}\label{sec:model}
In this section we list our assumptions and introduce the main results of the paper.
Let $\Omega\subset\R^n$ be an open set with Lipschitz boundary, let $1< p<+\infty$ and $f:[0,+\infty)\to[0,+\infty)$ a lower semicontinuous, increasing function satisfying
\begin{equation}\label{eq: alfabeta}
\lim_{t\to 0^+}\frac{f(t)}t=\alpha>0, \quad \lim_{t\to +\infty} f(t)=\beta>0\,.
\end{equation} 
{Let $\rho\in L^\infty(\R^n;[0,+\infty))$ be a convolution kernel. The minimal assumption is that
\begin{enumerate}[font={\normalfont},label={(N1)}]
\item $\rho$ is Riemann integrable with $\|\rho\|_1=1$ and $S=S_\rho:=\{x\in\R^n:\,\, \rho(x)\neq0\}$ is a bounded, open, convex and symmetrical set. \label{ass-N1}
\end{enumerate}
As every Riemann integrable function is continuous at almost
every point, we may also suppose, up to a modification on a null set, that $\rho$ is lower semicontinuous.  Also notice that, by a simple scaling argument, one can always consider the case of kernels with unit mass, up to modifying the constant $\alpha$ in \eqref{eq: alfabeta}.

A sequence $(\rho_\varepsilon)_{\varepsilon>0}$ of convolution nuclei is then obtained by setting, for every $x\in\R^n$ and $\varepsilon>0$, 
\begin{equation*}
\rho_\varepsilon(x):=\frac{1}{\varepsilon^n}\rho\left(\frac{x}{\varepsilon}\right)\,.
\end{equation*}
For every $v\in\R^n$ we define
\begin{equation}\
\phi_\rho(v):=2\sup \{|\langle y , v\rangle|:\,\, y\in S\}\,.
\label{eq:funphi}
\end{equation}
Under the previous assumptions on $S$, the function $\phi_\rho$ turns out to be a norm on $\R^n$.}

To obtain our main result, we will have to couple \ref{ass-N1} with the additional assumption that the convolution kernel is a non-increasing function with respect to the norm $|\cdot|_S$, that is
\begin{enumerate}[font={\normalfont},label={(N2)}] 
\item $|x|_S\ge |y|_S\Longrightarrow\rho(x)\le \rho(y)$ for all $x$, $y \in \R^n$. \label{ass-N2}
\end{enumerate}
Equivalently, we require that it exists a non-increasing function $\varrho\colon \R^+\to \R^+$ such that $\rho(x)=\varrho(|x|_S)$. {Notice that, in the case $S=B_1$, every non-increasing radial function $\rho$ complies with \ref{ass-N2}.}

Let $W:\R^{n\times n}\to\R$ be a convex positive function on the subspace $\mathbb{M}^{n\times n}_{sym}$ of symmetric matrices, such that
\begin{equation}\label{eq:coerciv}
 W({\bf 0})=0\,,\quad c|M|^p\le W(M)\leq C(1+|M|^p)\,.
\end{equation}
{For every $\varepsilon>0$ we consider the functional $F_\varepsilon: L^1(\Omega;\R^n)\to [0,+\infty]$ defined as
\begin{equation}
F_\varepsilon(u)=
\begin{cases}
\displaystyle\frac{1}{\varepsilon}\int_\Omega f\left(\varepsilon \int_{\Omega}W(\mathcal{E}u(y))\rho_\varepsilon(x-y)\,\mathrm{d}y\right)\,\mathrm{d}x, & \mbox{ if }u\in W^{1,p}(\Omega;\R^n)\,,\\
+\infty\,, & \mbox{ otherwise on $L^1(\Omega;\R^n)$. }
\end{cases}
\label{energies0}
\end{equation}}

We will deal with a localized version of the energies \eqref{energies0}. Namely, for every $A\in\mathcal{A}(\Omega)$, we will denote by $F_\varepsilon(u,A)$ the same functional as in \eqref{energies0} with the set $A$ in place of $\Omega$. When $A=\Omega$, we simply write $F_\varepsilon(u)$ in place of $F_\varepsilon(u, \Omega)$.

The following theorem is the  first main result of this paper. We notice that the additional assumption {\ref{ass-N2}} on the structure of the convolution kernel is required in (ii) below only to obtain the optimal lower bound for the bulk term of the energy.

\begin{thm}\label{thm:mainresult}
Let $\rho\in L^\infty(\R^n;[0,+\infty))$ be a convolution kernel as in {\rm\ref{ass-N1}}, and let $F_\varepsilon$ be defined as in \eqref{energies0}. Under assumptions \eqref{eq: alfabeta} and \eqref{eq:coerciv}, it holds that
\begin{enumerate}
\begin{item}[{\rm (i)}]
there exists a constant $c_0$ independent of $\varepsilon$ such that, for all $(u_\varepsilon)\subset L^p(\Omega;\R^n)$ satisfying
$
F_\varepsilon(u_\varepsilon)
\leq C
$ {for every $\varepsilon>0$},
one can find a sequence $\overline{u}_\varepsilon \in {GSBV^p(\Omega;\R^n)}$ with
\[
\begin{split}
&\overline{u}_\varepsilon-u_\varepsilon \to 0 \mbox{ in measure  on }\Omega \\
&F_\varepsilon(u_\varepsilon)\geq c_0 \left(\int_{\Omega} W(\mathcal{E}\overline{u}_\varepsilon)\,\mathrm{d}x + 2 \mathcal{H}^{n-1}(J_{\overline{u}_\varepsilon} \cap \Omega) \right)\,.
\end{split}
\]
\end{item}
\begin{item}[{\rm (ii)}]
If, in addition, $\rho$ complies with {\rm\ref{ass-N2}}, then the {functionals} $(F_\varepsilon)$ $\Gamma$\hbox{-}converge, as $\varepsilon\to0$, to the functional
\begin{equation}
F(u)=
\begin{cases}
\displaystyle\alpha\int_\Omega W(\mathcal{E}u)\,\mathrm{d}x + \beta\,\int_{J_u}\phi_\rho(\nu)\,\mathrm{d}\mathcal{H}^{d-1}\,\,, & \mbox{ if }u\in GSBD^p(\Omega)\cap L^1(\Omega;\R^n)\,,\\
+\infty\,, & \mbox{ otherwise on $L^1(\Omega;\R^n)$, }
\end{cases}
\label{limitfunct}
\end{equation}
with respect to the $L^1$ convergence in $\Omega$.
\end{item}
\end{enumerate}
\end{thm}

The $L^1$-convergence on the whole $\Omega$ can be enforced with the addition of a lower order fidelity term, as we have discussed in Remark \ref{rem: compactness}. This motivates the statement below, {where} we consider a continuous increasing function $\psi:[0,+\infty)\to[0,+\infty)$ such that
\begin{equation}\label{eq: hpsi}
\psi(0)=0,\quad \psi(s+t)\le C(\psi(s)+\psi(t)), \quad \psi(s)\le C(1+s^p), \quad \lim_{s\to+\infty}\frac{\psi(s)}{s}=+\infty
\end{equation} 
and we set for every $A\in\mathcal{A}(\Omega)$
\begin{equation}
G_\varepsilon(u, A)=
\begin{cases}
\displaystyle F_\varepsilon(u, A)+\int_A\psi(|u|)\,\mathrm{d}x, & \mbox{ if }u\in W^{1,p}(A;\R^n)\,,\\
+\infty\,, & \mbox{ otherwise on $L^1(A;\R^n)$. }
\end{cases}
\label{energies1}
\end{equation}
As before, we simply write $G_\varepsilon(u)$ in place of $G_\varepsilon(u, \Omega)$. 
Then we have the following result.

\begin{thm}\label{thm:mainresult2}
Under assumptions \eqref{eq: alfabeta}, {\rm\ref{ass-N1}}, \eqref{eq:coerciv},  and \eqref{eq: hpsi} it holds that
\begin{enumerate}
\begin{item}[{\rm (i)}]
If $(u_\varepsilon)\subset L^p(\Omega;\R^n)$ is such that
$
G_\varepsilon(u_\varepsilon)
\leq C
$ {for every $\varepsilon>0$},
then $(u_\varepsilon)$ is compact in $L^1(\Omega;\R^n)$.
\end{item}
\begin{item}[{\rm (ii)}]
If, in addition, {\rm\ref{ass-N2}} holds, the {functionals} $(G_\varepsilon)$ $\Gamma$\hbox{-}converge, as $\varepsilon\to0$, to the functional
\begin{equation*}
G(u)=
\begin{cases}
\displaystyle F(u) + \!\!\int_\Omega\psi(|u|)\,\mathrm{d}x\,, \!\!& \mbox{ if }u\in GSBD^p(\Omega)\cap L^1(\Omega;\R^n)\,,\\
+\infty\,, & \mbox{ otherwise on $L^1(\Omega;\R^n)$, }
\end{cases}
\end{equation*}
with respect to the $L^1$ convergence in $\Omega$.
\end{item}
\end{enumerate}
\end{thm}


\section{Compactness and estimate from below of the $\Gamma$\hbox{-}limit}\label{sec:compactnessestimbelow}

With the following proposition, we prove the compactness statements in Theorem~\ref{thm:mainresult}(i),  and {Theorem~\ref{thm:mainresult2}(i)}, respectively. These results can be easily inferred by a comparison with non-local integral energies whose densities are averages of the gradient on balls with small radii, for which a compactness result has been provided in \cite[Proposition~4.1]{SS21}. In order to do that, we will only require assumption \ref{ass-N1} on the convolution kernel $\rho$.
\begin{prop}\label{prop:compactness}
Let $A\in\mathcal{A}(\Omega)$, and let $F_\varepsilon$, $G_\varepsilon$ be defined as in \eqref{energies0}, and \eqref{energies1}, respectively, where $\rho\in L^\infty(\R^n;[0,+\infty))$ satisfies {\rm\ref{ass-N1}}. Then:
\begin{enumerate}
\begin{item}[{\rm (i)}]
Assume \eqref{eq: alfabeta}, \eqref{eq:coerciv}. If $(u_\varepsilon)\subset L^p(\Omega;\R^n)$ is such that
$
F_\varepsilon(u_\varepsilon, A)
\leq C
$ {for every $\varepsilon>0$},
one can find a sequence $\overline{u}_\varepsilon \in {GSBV^p(A;\R^n)}$ with
\[
\begin{split}
&\overline{u}_\varepsilon-u_\varepsilon \to 0 \mbox{ in measure  on }A \\
&F_\varepsilon(u_\varepsilon, A)\geq c_0 \left(\int_{A} W(\mathcal{E}\overline{u}_\varepsilon)\,\mathrm{d}x + 2 \mathcal{H}^{n-1}(J_{\overline{u}_\varepsilon} \cap A) \right)
\end{split}
\]
{for some $c_0>0$.}
\end{item}
\begin{item}[{\rm (ii)}]
Assume \eqref{eq: alfabeta}, \eqref{eq:coerciv}, and \eqref{eq: hpsi}. If $(u_\varepsilon)\subset L^p(\Omega;\R^n)$ is such that \, \,
$
G_\varepsilon(u_\varepsilon, A)
\leq C
$ {for every $\varepsilon>0$},
then $(u_\varepsilon)$ is compact in $L^1(A;\R^n)$.
\end{item}
\end{enumerate}
\end{prop}

\proof
Let $\eta\in(0,1)$ be fixed such that $B_\eta(0)\subset\subset S$, and denote by $m_\eta$ the minimum of $\rho$ on $\overline{B_\eta}$, which is strictly positive as we are assuming that $\rho>0$ on $S$. 
Setting $\tilde{f}(t):=f(m_\eta\omega_n\eta^nt)$ and for any $\varepsilon>0$, we consider the energies
\begin{equation}
\widetilde{F}_\varepsilon(u)=
\begin{cases}
\displaystyle\frac{1}{\varepsilon}\int_\Omega \tilde{f}\left(\varepsilon \dashint_{B_{\eta\varepsilon}(x)\cap\Omega}W(\mathcal{E}u(y))\,\mathrm{d}y\right)\,\mathrm{d}x, & \mbox{ if }u\in W^{1,p}(\Omega;\R^n)\,,\\
+\infty\,, & \mbox{ otherwise on $L^1(\Omega;\R^n)$. }
\end{cases}
\label{energies2}
\end{equation}
Since $B_\eta(0)\subseteq S$ and $\rho\geq m_\eta$ on $B_\eta(0)$, a simple computation shows that
\begin{equation}
\widetilde{F}_{\varepsilon}(u , A) \leq F_{\varepsilon}(u , A)
\label{eq:comparison}
\end{equation}
for every $u\in W^{1,p}(\Omega;\R^n)$ and $A\subseteq\Omega$ open set.
By virtue of \eqref{eq:comparison}, to obtain (i) it will suffice to apply the argument of \cite[Proposition~4.1]{SS21} to the sequence $\widetilde{F}_\varepsilon$ in \eqref{energies2}. We then omit the details. 

We now come to (ii). If additionally $
G_\varepsilon(u_\varepsilon, A)
\leq C
$, following the argument for \cite[Proposition~4.1(ii)]{SS21}, it can be shown that the sequence $(\bar{u}_\epsilon)$ constructed in (i) complies with
\begin{equation*}
\int_{A}\psi(|\overline{u}_\varepsilon(x)|)\,\mathrm{d}x + \int_{A}|\mathcal{E}\overline{u}_\varepsilon(x)|^p\,\mathrm{d}x + \mathcal{H}^{n-1}(J_{\overline{u}_\varepsilon}\cap A)\le C<+\infty
\end{equation*}
for all $\e$.
Therefore, in view of the growth assumption \eqref{eq: hpsi} on $\psi$, Theorem \ref{th: GSDBcompactness} and Remark \ref{rem: compactness} apply, and this provides the compactness of the sequence $(\overline{u}_\varepsilon)$ in $L^1(A;\R^n)$. Then, since $\overline{u}_\varepsilon-u_\varepsilon \to 0$ in measure  on $A$, with the Vitali dominated convergence Theorem we infer that $(u_\varepsilon)$ is compact in $L^1(A;\R^n)$ as well. This concludes the proof of (ii).
\endproof

Now, we turn to provide a first estimate of the $\Gamma$-liminf of the functionals $F_\e$. This estimate is optimal, up to a small error, only for the bulk part of the energy, and this is the only very point where we need to require the additional assumptions \ref{ass-N2} on the convolution kernels (see Section~\ref{sec:estimbelowbulk}). The proof of an optimal estimate for the surface term, instead, will be derived separately by means of a slicing argument (see Proposition \ref{prop:lowboundjump} below) for more general kernels complying only with \ref{ass-N1} providing the comparison estimate \eqref{eq:comparison}. As the two parts of the energy are mutually singular, the localization method of Lemma \ref{lem: lemmasup} will eventually allow us to get the $\Gamma$-liminf inequality.

\subsection{Estimate from below of the bulk term}\label{sec:estimbelowbulk}
We begin by giving the announced estimate for the bulk term.

\begin{prop}\label{prop:estimate}
Let $A\in\mathcal{A}(\Omega)$ with $A\subset\subset \Omega$, and consider a sequence $u_\e\in W^{1,p}(\Omega;\R^n)$ converging to $u$ in $L^1(\Omega; \R^n)$.
Assume \eqref{eq: alfabeta} and \eqref{eq:coerciv}, let $\eta\in(0,1)$ be fixed and let $\rho$ comply with (N1)--(N2). 
Suppose that
\begin{equation}
\sup_{\varepsilon>0}F_\varepsilon(u_\varepsilon, A)\leq C\,.
\label{eq:equibounded}
\end{equation}
Then, for every fixed $0<\delta<1$, there exist a constant $M_{\delta,\eta}$ only depending on $f$, $\delta$ and $\eta$, {a constant $\sigma_\eta$ depending on $\rho, \eta$ such that $\sigma_\eta\to0$ as $\eta\to0$}, and a sequence of functions $(v_\varepsilon^{\delta,\eta})\subset {GSBV^p(A;\R^n)}$ such that
\begin{itemize}
\item[{\rm (i)}] $\displaystyle \alpha(1-\sigma_\eta)^2(1-\delta)^{2n+1}\int_{A} W(\mathcal{E}v_\varepsilon^{\delta,\eta}(x))\,\mathrm{d}x\leq F_\varepsilon(u_\e,A)$;
\item[{\rm (ii)}] $\displaystyle \mathcal{H}^{n-1}(J_{v_\varepsilon^{\delta,\eta}})\leq M_{\delta,\eta}\, F_\varepsilon(u_\e,A)$;
\item[(iii)] $\displaystyle v_\varepsilon^{\delta,\eta} \to u$ in $L^1(A; \R^n)$ as $\varepsilon \to 0$.
\end{itemize} 
\end{prop}

\proof
We first consider the case $f(t)=\min\{at,b\}$, with $a,b>0$. For given $\eta$, we introduce the truncated kernel $\rho^\eta(x):=\frac{1}{1-\sigma_\eta}\rho(x)(1-\chi_{\eta S}(x))$, where the constant $\sigma_\eta$ is given by
\begin{equation*}
\sigma_\eta:=\int_{\eta S} \rho(x)\,\mathrm{d}x\,,
\end{equation*}
and $\sigma_\eta\to0$ as $\eta\to0$.  Notice that with this choice of $\sigma_\eta$ one has $\int_{\R^n}\rho^\eta(x)\,\mathrm{d}x=1$.

For fixed $\delta\in(0,1)$, we then define
\begin{equation}\label{cdeltaeta}
C_{\delta,\eta}:=\frac{1}{(1-\delta)^n(1-\sigma_\eta)}\,,
\end{equation}
and the functions
\begin{equation*}
\begin{split}
\psi_\varepsilon^{\eta,\delta}(x)  & := \varepsilon \int_{\Omega}W(\mathcal{E} u_\varepsilon(y))\rho^\eta_{\varepsilon(1-\delta)}(y-x)\,\mathrm{d}y \\
\psi_\varepsilon(x) & := \varepsilon \int_{\Omega}W(\mathcal{E} u_\varepsilon(y))\rho_{\varepsilon}(y-x)\,\mathrm{d}y\,.
\end{split}
\end{equation*}
Observe that, since $W\ge 0$, by the definition of $\rho^\eta$ and assumption (N2), we get
\begin{equation}\label{eq:elementary}
\psi_\varepsilon^{\eta,\delta}(x) \le C_{\delta,\eta} \psi_\varepsilon(x) 
\end{equation}
for all $x\in A$. Define now the following sets, depending on $\delta,\eta$ and $S$:
\begin{equation}
K_\varepsilon:=\left\{x\in A:\,\, \psi_\varepsilon^{\eta,\delta}(x) \geq {C_{\delta,\eta}}\, \frac{b}{a}\right\}\,,
\label{eq:keps}
\end{equation}
\begin{equation}
K^{\prime }_\varepsilon:=\bigl\{x\in A:\,\, {\rm dist}_S(x,K_\varepsilon)\le \delta\eta\epsilon\bigr\}\,.
\label{eq:keps2}
\end{equation}
We prove the inclusion
\begin{equation}
K^{\prime}_\varepsilon \subseteq \left\{x\in A:\,\, \psi_\varepsilon(x)\geq \, \frac{b}{a}\right\}\,.
\label{eq:inclusion}
\end{equation}
For this, if $x\in K^{\prime}_\varepsilon$ then there exists $z\in K_\varepsilon$ such that $|x-z|_S\leq \delta\eta\varepsilon$. Now, by the triangle inequality, for every $y\in\Omega$ it holds that
\begin{equation*}
\frac{|x-y|_S}{\varepsilon}\leq \delta\eta + \frac{|z-y|_S}{\varepsilon}\,,
\end{equation*}
whence
\begin{equation*}
\frac{|x-y|_S}{\varepsilon}\leq  \frac{|z-y|_S}{(1-\delta)\varepsilon}
\end{equation*}
if and only if ${|z-y|_S}\geq (1-\delta)\eta\varepsilon$. In this case, since $\rho$ is non-increasing with respect to $|\cdot|_S$, we have
\begin{equation*}
\rho\left(\frac{y-x}{\varepsilon}\right) \geq (1-\sigma_\eta) \rho^\eta\left(\frac{y-z}{(1-\delta)\varepsilon}\right)\,.
\end{equation*}
Notice that this inequality holds true also if $|z-y|_S<(1-\delta)\eta\varepsilon$. In this case, indeed, one has
$\frac{y-z}{(1-\delta)\varepsilon}\in \eta S$ and hence $\rho^\eta\left(\frac{y-z}{(1-\delta)\varepsilon}\right)=0$ by definition of $\rho^\eta$. Rescaling the kernels and using \eqref{cdeltaeta} we get 
\[
\rho^\eta_{\varepsilon(1-\delta)}(y-z)\le C_{\delta,\eta} \rho_{\varepsilon}(y-x)\,,
\] 
so that
\begin{equation*}
\psi_\varepsilon(x)\geq \frac{\psi_\varepsilon^{\eta,\delta}(z)}{C_{\delta,\eta}}\geq \frac{b}{a}
\end{equation*}
and the proof of \eqref{eq:inclusion} is concluded.

Now, from the inclusion \eqref{eq:inclusion} and the fact that ${f}(t)= b$ for $t\geq \frac{b}{a}$, we deduce that
\begin{equation}
\mathcal{L}^n(K^{\prime}_\varepsilon)\leq \frac{\varepsilon}{b} {F}_\varepsilon(u_\varepsilon, A)\,.
\label{stima1}
\end{equation}
Applying the coarea formula (see for instance \cite[Theorem 3.14]{EvansGariepy92}) to the $1$-Lipschitz function \, $g(x):= \mathrm{dist}_S(x, K_\epsilon)$ in the open set $\{0<g(x)<\eta\delta \epsilon\}\subset K^{\prime}_\varepsilon$ we get
\[
 \frac{\varepsilon}{b} F_\varepsilon(u_\varepsilon, A) \geq \mathcal{L}^n(K^{\prime}_\varepsilon)\ge \int_{0}^ {\eta\delta \epsilon}\mathcal{H}^{n-1}(\{g=t\})\,\mathrm{d}t\,. 
\]
It follows that 
we can choose $0<\delta^\prime_\e < \eta\delta\e$ such that, for 
\begin{equation}
K^{\prime\prime}_\varepsilon:=\{x\in A:\,\, {\rm dist}_S(x,K_\varepsilon)\le \delta^\prime_\e\}\,,
\label{eq:kpeps}
\end{equation}
it holds
\begin{equation}
\mathcal{H}^{n-1}(\partial K^{\prime\prime}_\varepsilon)=\mathcal{H}^{n-1}(\{x\in A:\,\,{\rm dist}_S(x,K_\varepsilon)= \delta^\prime_\e\})\leq \frac{1}{\eta\delta b}{F}_\varepsilon(u_\varepsilon,  A)\,.
\label{eq:boundmennucci}
\end{equation}

We define a sequence $(v_\varepsilon^{\delta,\eta})$ of functions in ${GSBV^p(A;\R^n)}$ as
\begin{equation}
v_\varepsilon^{\delta,\eta}(x):=
\begin{cases}
\displaystyle\int_{\Omega} u_\varepsilon(y)\rho^\eta_{(1-\delta)\varepsilon}(y-x)\,\mathrm{d}y & \mbox{ if }x\in A\backslash K^{\prime\prime}_\varepsilon\,,\\
0 & \mbox{ otherwise. }
\end{cases}
\label{eq:vdelta}
\end{equation}
Since $\|\rho^\eta\|_1=1$, by Proposition~\ref{prop:convolutions}(ii) (applied for $\theta_\varepsilon=\varepsilon(1-\delta)$) and the fact that, by construction and  \eqref{stima1}, it holds $\mathcal L^n(K^{\prime\prime}_\varepsilon)\to 0$ when $\e \to 0$, we have that $\displaystyle v_\varepsilon^{\delta,\eta} \to u$ in $L^1(A; \R^n)$ as $\varepsilon \to 0$. We also have $\mathcal{H}^{n-1}(J_{v_\varepsilon^{\delta,\eta}})\le \mathcal{H}^{n-1}(\partial K^{\prime\prime}_\varepsilon)$, so that with \eqref{eq:boundmennucci} we deduce (ii) for $M_{\delta,\eta}=\frac1{\eta\delta b}$.

Now, since $K_\e \subset K^{\prime\prime}_\e$ and $A\subset\subset \Omega$, it holds
$
\psi_\varepsilon^{\eta,\delta}(x)< {C_{\delta,\eta}}\frac ba
$
for all $x\in  K^{\prime\prime}_\e$. As $f(t)=\min\{at, b\}$, this gives
\begin{equation}\label{eq: banale}
f(\psi_\varepsilon^{\eta,\delta}(x))\geq \frac{a}{C_{\delta,\eta}} \psi_\varepsilon^{\eta,\delta}(x)
\end{equation}
for all $x \in A \setminus K^{\prime\prime}_\e$. Now, since the function $f$ is concave and $f(0)=0$, it holds $f(\lambda t) \geq \lambda f(t)$ for all $\lambda\in[0,1]$. Combining with the monotonicity of $f$ and \eqref{eq:elementary}, we have
\begin{equation}
f\left(\psi_\varepsilon(x)\right)\ge \frac{1}{C_{\delta,\eta}} f\left(\psi_\varepsilon^{\eta,\delta}(x)\right)
\end{equation}
for all $x \in A$. With this, using \eqref{eq: banale}, the Jensen's inequality, 
\eqref{commut}, \eqref{eq:vdelta}, and since $W(\mathbf 0)=0$, we get
\begin{equation*}
\begin{split}
F_\varepsilon(u_\varepsilon,A) & \geq \frac{1}{\varepsilon}\int_{A\backslash K^{\prime\prime}_\varepsilon}f\left(\psi_\varepsilon(x)\right)\,\mathrm{d}x\\
& \geq \frac{1}{\varepsilon C_{\delta,\eta}}\int_{A\backslash K^{\prime\prime}_\varepsilon}f\left(\psi_\varepsilon^{\eta,\delta}(x)\right)\,\mathrm{d}x \geq \frac{a}{\varepsilon C_{\delta,\eta}^2}\int_{A\backslash K^{\prime\prime}_\varepsilon}\psi_\varepsilon^{\eta,\delta}(x)\,\mathrm{d}x\\
& \geq \frac{a}{C_{\delta,\eta}^2}\int_{A\backslash K^{\prime\prime}_\varepsilon}W\left(\int_{\Omega}\mathcal{E} u_\varepsilon(y)\rho^\eta_{(1-\delta)\varepsilon}(y-x)\,\mathrm{d}y\right)\,\mathrm{d}x\\
& = \frac{a}{C_{\delta,\eta}^2}\int_{A\backslash K^{\prime\prime}_\varepsilon}W(\mathcal{E}v_\epsilon^{\delta,\eta}(x))\,\mathrm{d}x \\
&= a(1-\sigma_\eta)^2(1-\delta)^{2n}\int_{A}W(\mathcal{E}v_\epsilon^{\delta,\eta}(x))\,\mathrm{d}x\,.
\end{split}
\end{equation*}
For a general $f$ complying with \eqref{eq: alfabeta}, use Lemma \ref{lem: belowapprox} to find $a_\delta, b_\delta>0$ with $a_\delta\ge \alpha(1-\delta)$ and $f(t) \ge \min\{a_\delta t, b_\delta\}$ for all $t \in \R$, and perform the same construction as in the previous step. This gives (iii), (ii) (with $M_\delta:=\frac1{\delta\eta b_\delta}$) and
\[
\begin{split}
F_\varepsilon(u_\varepsilon,A) & \geq a_\delta(1-\sigma_\eta)^2(1-\delta)^{2n}\int_{A}W(\mathcal{E}v_\epsilon^{\delta,\eta}(x))\,\mathrm{d}x \\
& \geq \alpha(1-\sigma_\eta)^2(1-\delta)^{2n+1}\int_{A}W(\mathcal{E}v_\epsilon^{\delta,\eta}(x))\,\mathrm{d}x\,,
\end{split}
\]
that is (i).
\endproof

\subsection{Estimate from below of the surface term}\label{sec:estimbelowsurf}

In this section we derive by slicing a lower bound for the surface term in the energy. It is worth mentioning that, by virtue of \eqref{eq:Gxi}, the desired estimate could be probably also obtained by adapting to the $GSBD$-setting the semi-discrete approach of \cite[Proposition~6.4]{Negri2006}. Nonetheless, that argument is quite delicate for our purposes, and more complicated than we need. It indeed aimed to provide an optimal lower bound for both the bulk and the surface terms in a unique proof by means of a slicing procedure. In our case, the general form of the bulk energy we are considering does not comply with slicing arguments.  Therefore, on the one hand, the two terms have to be estimated separately. On the other hand, an independent and simpler strategy can be followed to provide a lower bound with optimal constant in front of the surface energy. 

We set
\begin{equation}
\tau_\xi :=\mathcal{H}^1(\{x\in S:\,\, x=t\xi \mbox{ for }t\in\R\})\,,
\label{eq:tauxi}
\end{equation}
for every $\xi\in\mathbb{S}^{n-1}$.

\begin{prop}\label{prop:lowboundjump} 
Let $\rho\in L^\infty(\R^n;[0,+\infty))$ be a convolution kernel complying with {\rm\ref{ass-N1}}, and let $F_\varepsilon$ be defined as in \eqref{energies0}. Assume \eqref{eq: alfabeta} and \eqref{eq:coerciv}. Let $\delta\in(0,1)$ be fixed, and consider a sequence $\e_j \to 0$. Let $A\in\mathcal{A}(\Omega)$ and $u_j\in W^{1,p}(A;\R^n)$ converging to $u$ in $L^1(A; \R^n)$. Assume that 
\[
\mathop{\lim\inf}_{j\to+\infty}F_{\varepsilon_j}(u_j,A)<+\infty\,.
\]
Then $u \in GSBD^p(A)$ and 
\begin{equation}
\displaystyle\mathop{\lim\inf}_{j\to+\infty}F_{\varepsilon_j}(u_j,A)\geq \beta(1-\delta)\int_{J_{u}^\xi\cap A}\tau_\xi|\langle \nu,\xi\rangle|\,\mathrm{d}\mathcal{H}^{n-1}
\label{eq:lowboundjump}
\end{equation}
for every $\xi\in \mathbb{S}^{n-1}$.
\end{prop}

\proof
It follows from Proposition \ref{prop:compactness} and Theorem \ref{th: GSDBcompactness} that $u \in GSBD^p(A)$.
To prove \eqref{eq:lowboundjump}, we first note that, by virtue of the growth assumption \eqref{eq:coerciv}, we have
\begin{equation*}
W(\mathcal{E}u)\geq c |\mathcal{E}u|^p\geq c |\langle (\mathcal{E}u)\xi,\xi\rangle|^p\,,
\end{equation*}
for every $\xi\in \mathbb{S}^{n-1}$. 
Thus, for every fixed $\xi$, since $f$ is non-decreasing, it will be sufficient to provide a lower estimate for the energies
\begin{equation}
F_{\varepsilon_j}^\xi(u_j,A) := \displaystyle\frac{1}{\varepsilon_j}\int_{A}f\left({c}\varepsilon_j \int_{S_{{\varepsilon_j}}(x)}|\langle (\mathcal{E}u_j(z))\xi,\xi\rangle|^p\rho_{\varepsilon_j}(z-x)\,\mathrm{d}z\right)\,\mathrm{d}x\,.
\label{eq:Gxi}
\end{equation}
We proceed by a slicing argument, and for each $x\in A$ we denote by $x_\xi$ and $y_\xi$ the projections of $x$ onto $\Xi$ and $\Pi^\xi$, respectively. Since $S$ is open and convex, for every fixed $\xi\in \mathbb{S}^{n-1}$ we can find a radius $r=r(\delta,S)>0$ such that the cylinder
\begin{equation}
C_{(1-\delta),r}^\xi=\left(-\lambda_{\xi,\delta},\lambda_{\xi,\delta}\right)\times B^{n-1}_{r}(0) \subset\subset S\,,
\end{equation}
where $\lambda_{\xi,\delta}:=\frac{\tau_\xi}{2}(1-\delta)$ and $\tau_\xi$ is the length of the section ${S_\xi}$. Indeed, since $S$ is open, some $\eta>0$ can be found such that $\overline{B_{\eta}(0)}$ is contained in $S$. Now, if $t=(1-\delta)s$ for some $s\in {S_\xi}$ and $y\in\xi^\perp$ with $|y|\leq\eta$, then $t\xi+\delta y\in S$ from the convexity of $S$. Thus, it will suffice to choose $r:=\delta\eta$. 

If we denote by $m_C$ the minimum of $\rho$ on $\overline{C_{(1-\delta),r}^\xi}$, we then have
\begin{equation} 
\begin{split}
& F_{\varepsilon_j}^\xi(u_j,A)\\
 & =\displaystyle\int_{\Pi^\xi}\mathrm{d}\mathcal{H}^{n-1}(y_\xi)\left(\frac{1}{\varepsilon_j}\int_{A_{\xi,y_\xi}}f\left({c}\varepsilon_j \int_{S_{{\varepsilon_j}}(x)}|\langle (\mathcal{E}u_j(z))\xi,\xi\rangle|^p\rho_{\varepsilon_j}(z-x)\,\mathrm{d}z\right)\,\mathrm{d}x_\xi\right) \\
 & \geq\displaystyle\int_{\Pi^\xi}\mathrm{d}\mathcal{H}^{n-1}(y_\xi)\left(\frac{1}{\varepsilon_j}\int_{A_{\xi,y_\xi}}\tilde{f}\left(\frac{1}{\varepsilon_j^{n-1}}\int_{C_{(1-\delta){\varepsilon_j},r\varepsilon_j}^\xi(x)}|\langle (\mathcal{E}u_j(z))\xi,\xi\rangle|^p\,\mathrm{d}z\right)\,\mathrm{d}x_\xi\right)\,,
\end{split}
\label{stimone0}
\end{equation}
where $\tilde{f}(t):=f(c\,m_C t)$. Note that $\tilde{f}(t)\to\beta$ as $t\to+\infty$.

We now set
\begin{equation*}
F_{\varepsilon_j}^{\xi,y_\xi}(u_j,A_{\xi,y_\xi}):= \frac{1}{\varepsilon_j}\int_{A_{\xi,y_\xi}}\tilde{f}\left(\frac{1}{\varepsilon_j^{n-1}}\int_{C_{(1-\delta){\varepsilon_j},r\varepsilon_j}^\xi(x)}|\langle (\mathcal{E}u_j(z))\xi,\xi\rangle|^p\,\mathrm{d}z\right)\,\mathrm{d}x_\xi\,.
\end{equation*}
We denote (with a slight abuse of notation) still with $z$ the $(n-1)$-dimensional variable in $B^{n-1}_{r\epsilon_j}(y_\xi)$. Set 
${w}_j^{\xi,y_\xi}(t):=\dashint_{B^{n-1}_{r\epsilon_j}(y_\xi)}\langle u_j(z+t\xi)),\xi\rangle\,\mathrm{d}z$.

By virtue of Lemma~\ref{lem:lbp}(ii), applied with $\theta_{\varepsilon_j}=r\varepsilon_j$, we have that ${w}_j^{\xi,y_\xi}$ converges to $u^{\xi,y_\xi}$ in $L^1(A_{\xi,y_\xi})$ for a.e. $y_\xi$. Furthermore, setting $g(t):=\tilde{f}(\omega_{n-1}r^{n-1} t)$, Fubini's Theorem, Jensen's inequality and the monotonicity of $\tilde{f}$ entail that
\begin{align}\label{stimon}
\begin{split}
&F_{\varepsilon_j}^{\xi,y_\xi}(u_j,A_{\xi,y_\xi})\\
&= \frac{1}{\varepsilon_j}\int_{A_{\xi,y_\xi}}\tilde{f}\left(\frac{1}{\varepsilon_j^{n-1}} \int_{B^{n-1}_{r\epsilon_j}(y_\xi)}\mathrm{d}z\int_{x_\xi-\lambda_{\xi,\delta}\varepsilon_j}^{x_\xi+\lambda_{\xi,\delta}\varepsilon_j}|\langle (\mathcal{E}u_j(z+t\xi))\xi,\xi\rangle|^p\,\mathrm{d}t\right)\,\mathrm{d}x_\xi\\
&= \displaystyle\frac{1}{\varepsilon_j}\int_{A_{\xi,y_\xi}}\tilde{f}\left(\frac{1}{\varepsilon_j^{n-1}} \int_{x_\xi-\lambda_{\xi,\delta}\varepsilon_j}^{x_\xi+\lambda_{\xi,\delta}\varepsilon_j}\left(\int_{B^{n-1}_{r\epsilon_j}(y_\xi)}|\langle (\mathcal{E}u_j(z+t\xi))\xi,\xi\rangle|^p\,\mathrm{d}z\right)\mathrm{d}t\right)\,\mathrm{d}x_\xi \\
& \geq \displaystyle\frac{1}{\varepsilon_j}\int_{A_{\xi,y_\xi}}\tilde{f}\left(\omega_{n-1}r^{n-1}\int_{x_\xi-\lambda_{\xi,\delta}\varepsilon_j}^{x_\xi+\lambda_{\xi,\delta}\varepsilon_j}\left(\dashint_{B^{n-1}_{r\epsilon_j}(y_\xi)}\langle (\mathcal{E}u_j(z+t\xi))\xi,\xi\rangle\,\mathrm{d}z\right)^p\mathrm{d}t\right)\,\mathrm{d}x_\xi\\
&= \displaystyle\frac{1}{\varepsilon_j}\int_{A_{\xi,y_\xi}}g\left(\int_{x_\xi-\lambda_{\xi,\delta}\varepsilon_j}^{x_\xi+\lambda_{\xi,\delta}\varepsilon_j}|\dot{w}_j^{\xi,y_\xi}(t)|^p\,\mathrm{d}t\right)\,\mathrm{d}x_\xi\\
&= \lambda_{\xi,\delta}\displaystyle\frac{1}{\lambda_{\xi,\delta}\varepsilon_j}\int_{A_{\xi,y_\xi}}g\left(\int_{x_\xi-\lambda_{\xi,\delta}\varepsilon_j}^{x_\xi+\lambda_{\xi,\delta}\varepsilon_j}|\dot{w}_j^{\xi,y_\xi}(t)|^p\,\mathrm{d}t\right)\,\mathrm{d}x_\xi\,,
\end{split}
\end{align}

Now, for  the function $t\mapsto g(t)$ it still holds  $g(t)\to \beta$ when $t\to +\infty$. Hence, applying Theorem~\ref{thm:braides} 
to the one-dimensional energies
\begin{equation*}
\widetilde{F}_{\epsilon_j}^{\xi,y_\xi}({w}_j^{\xi,y_\xi}, A_{\xi,y_\xi}):=\displaystyle\frac{1}{\lambda_{\xi,\delta}\varepsilon_j}\int_{A_{\xi,y_\xi}}g\left(\int_{x_\xi-\lambda_{\xi,\delta}\varepsilon_j}^{x_\xi+\lambda_{\xi,\delta}\varepsilon_j}|\dot{w}_j^{\xi,y_\xi}(t)|^p\,\mathrm{d}t\right)\,\mathrm{d}x_\xi
\end{equation*}
we obtain the lower bound
\begin{equation}
\mathop{\lim\inf}_{j\to+\infty} \widetilde{F}_{\epsilon_j}^{\xi,y_\xi}({w}_j^{\xi,y_\xi}, A_{\xi,y_\xi}) \geq  2\beta\#(J_{u^{\xi,y_\xi}}\cap A_{\xi,y_\xi})\,.
\label{stima1D}
\end{equation}
Therefore, using \eqref{stimon} and \eqref{stima1D} we deduce
\begin{equation*}
\begin{split}
\mathop{\lim\inf}_{j\to+\infty} F_{\varepsilon_j}^{\xi,y_\xi}(u_j, A_{\xi,y_\xi})&\geq \lambda_{\xi,\delta} \mathop{\lim\inf}_{j\to+\infty} \widetilde{F}_{\epsilon_j}^{\xi,y_\xi}({w}_j^{\xi,y_\xi}, A_{\xi,y_\xi})\\
& \geq \beta\tau_\xi(1-\delta)\#(J_{u^{\xi,y_\xi}}\cap A_{\xi,y_\xi})\,.
\end{split}
\end{equation*}
With \eqref{stimone0} and Fatou's Lemma we finally have
\begin{equation*}
\begin{split}
\displaystyle\mathop{\lim\inf}_{j\to+\infty}F_{\varepsilon_j}(u_j,A)&\geq \mathop{\lim\inf}_{j\to+\infty}\displaystyle\int_{\Pi^\xi}F_{\varepsilon_j}^{\xi,y_\xi}(u_j, A_{\xi,y_\xi})\,\mathrm{d}\mathcal{H}^{d-1}(y_\xi) \\
& \geq \displaystyle\int_{\Pi^\xi} \left(\mathop{\lim\inf}_{j\to+\infty} F_{\varepsilon_j}^{\xi,y_\xi}(u_j, A_{\xi,y_\xi})\right)\, \mathrm{d}\mathcal{H}^{d-1}(y_\xi) \\
& \geq \beta\tau_\xi(1-\delta)\displaystyle\int_{\Pi^\xi} \#(J_{u^{\xi,y_\xi}}\cap A_{\xi,y_\xi})\, \mathrm{d}\mathcal{H}^{d-1}(y_\xi) \\
&=\beta\tau_\xi(1-\delta)\int_{J_u^\xi\cap A}|\langle \nu_u,\xi\rangle|\,\mathrm{d}\mathcal{H}^{d-1} \,,
\end{split}
\end{equation*}
{where in the last equality we used the Area Formula. This concludes the proof of \eqref{eq:lowboundjump}.}
\endproof

\subsection{Proof of the $\Gamma$\hbox{-}liminf inequality} \label{sec:gammaliminf}
{For any $A\in\mathcal{A}(\Omega)$, we denote by $F'(u, A)$ and $G'(u,A)$ the lower $\Gamma$\hbox{-}limits of $F_\varepsilon(u,A)$ and $G_\varepsilon(u,A)$, respectively, as defined in \eqref{eq:Gammaliminf}.} {It holds that $G'(u, A)\geq F'(u, A)$ for each $A\in\mathcal{A}(\Omega)$ and $u \in L^1(A;\R^n)$ (see, e.g., \cite[Proposition~6.7]{DM93}).}
The results of the previous subsection lead to the following estimate.
\begin{prop}\label{prop:boundsp}
Assume \eqref{eq: alfabeta}, \eqref{eq:coerciv},  and \eqref{eq: hpsi}. Let $F_\e$ and $G_\e$ be defined as in \eqref{energies0} and \eqref{energies1}, respectively, and let $\rho$ comply with \ref{ass-N1}-\ref{ass-N2}.  Let $u\in L^1(\Omega;\R^n)$, $A\in\mathcal{A}(\Omega)$, and define $F'(u, A)$ and $G'(u,A)$ by \eqref{eq:Gammaliminf} {in correspondence of $F_\varepsilon$ and $G_\varepsilon$, respectively}. If $F'(u,A)<+\infty$, then $u\in GSBD^p(A)$ and
\begin{itemize}
\item[{\rm (i)}] $\displaystyle  F'(u,A)\geq \alpha \int_A W(\mathcal{E}u)\,\mathrm{d}x$\,,
\item[{\rm (ii)}] $\displaystyle G'(u,A)\geq F'(u,A)\geq \beta \int_{J_u^\xi\cap A}\tau_\xi|\langle \nu_u,\xi\rangle|\,\mathrm{d}\mathcal{H}^{d-1}$
\end{itemize}
for every $\xi\in \mathbb{S}^{n-1}$.  If in addition $G'(u, A)<+\infty$ holds, then one also has
\begin{itemize}
\item[{\rm (iii)}] $\displaystyle G'(u,A)\geq  \alpha \int_A W(\mathcal{E}u)\,\mathrm{d}x+ \int_A \psi(|u|)\,\mathrm{d}x$.
\end{itemize}
\end{prop}

\proof
With \eqref{eq:Gammaliminf} and a diagonal argument, one may find (not relabeled) subsequences $(u_j)$ and $(\tilde{u}_j)$} converging to $u$ in $L^1(A;\R^n)$ such that
\begin{equation*}
F'(u,A)=\mathop{\lim\inf}_{j\to+\infty} F_{\varepsilon_j}(u_j,A)\,, \quad {G'(u,A)=\mathop{\lim\inf}_{j\to+\infty} G_{\varepsilon_j}(\tilde{u}_j,A)}\,.
\end{equation*}
With the first equality and Proposition \ref{prop:lowboundjump} we have that, if $F'(u,A)<+\infty$, then $u\in GSBD^p(A)$. By the second one, the superadditivity of the liminf, Fatou's lemma and \eqref{eq:Gammaliminf}, we have
\begin{equation*}
\begin{split}
G'(u,A)&=\mathop{\lim\inf}_{j\to+\infty} {G_{\varepsilon_j}(\tilde{u}_j,A)\geq \mathop{\lim\inf}_{j\to+\infty} F_{\varepsilon_j}(\tilde{u}_j,A)+ \mathop{\lim\inf}_{j\to+\infty} \int_A \psi(|\tilde{u}_j|)\,\mathrm{d}x} \\
&\geq F'(u,A)+\int_A \psi(|u|)\,\mathrm{d}x\,.
\end{split}
\end{equation*}
Hence, (iii) will follow once we have proved (i).

We therefore only have to check (i) and (ii). To this aim, let $\eta,\delta\in(0,1)$ be fixed. Then, by applying Proposition~\ref{prop:estimate} to the sequence $(u_j)$, we can find a sequence $(v_j^{\delta,\eta})\subset {GSBV^p(A;\R^n)}$ such that  $v_j^{\delta,\eta}\to u$ in $L^1(A)$ as $\varepsilon_j \to 0$ and
\begin{itemize}
\item[(a)] $\displaystyle \alpha (1-\sigma_\eta)^2(1-\delta)^{2n+1}\int_{A} W(\mathcal{E}v_j^{\delta,\eta}(x))\,\mathrm{d}x\leq F_{\varepsilon_j}(u_j,A)$;
\item[(b)] $\displaystyle \mathcal{H}^{d-1}(J_{v_j^{\delta,\eta}}\cap A)\leq M_{\delta,\eta}  F_{\varepsilon_j}(u_j,A)$.
\end{itemize} 
Now, the equiboundedness of $F_{\varepsilon_j}(u_j,A)$ combined with the bounds (a) and (b) allows to apply the lower semicontinuity part of Theorem \ref{th: GSDBcompactness} to the sequence $(v_j^{\delta,\eta})$. Taking into account that ${A^\infty}=\emptyset$ because $u\in L^1(A; \R^n)$, by the convexity of $W$ and \eqref{eq: GSBD comp}, (ii), we have
\begin{equation*}
\begin{split}
\alpha(1-\sigma_\eta)^2(1-\delta)^{2n+1}\int_{A} W(\mathcal{E}u(x))\,\mathrm{d}x & \leq \mathop{\lim\inf}_{j\to+\infty}\int_{A} W(\mathcal{E}v_j^{\delta,\eta}(x))\,\mathrm{d}x \\
& \leq \mathop{\lim\inf}_{j\to+\infty}F_{\varepsilon_j}(u_j,A)=F'(u, A)\,.
\end{split}
\end{equation*}
We then obtain (i) by letting $\delta\to0$ and $\eta\to0$ above.

For what concerns (ii), from \eqref{eq:lowboundjump} of Proposition~\ref{prop:lowboundjump} we get
\[
\beta(1-\delta)\int_{J_{u}^\xi\cap A}\tau_\xi|\langle \nu_u,\xi\rangle|\,\mathrm{d}\mathcal{H}^{n-1} \leq \mathop{\lim\inf}_{j\to+\infty} F_{\varepsilon_j}(u_j,A)=F'(u,A)
\]
for every $\xi\in \mathbb{S}^{n-1}$, so that (ii) follows by taking the limit as $\delta\to0$ again.
\endproof

For the proof of the $\Gamma$-liminf inequality, we need the following lemma, which can be found in \cite[Lemma~4.5]{Negri2006}.
\begin{lem}
Let $S\subset\R^n$ be a bounded, convex and symmetrical set, and let $\phi_\rho$ and $\tau_\xi$ be defined as in \eqref{eq:funphi} and \eqref{eq:tauxi}, respectively. Then
\begin{equation}
\phi_\rho(v)=\sup_{\xi\in \mathbb{S}^{n-1}} \tau_\xi |\langle v , \xi \rangle|\,.
\end{equation}  
\label{lem:negri1}
\end{lem}

We are now in a position to prove the $\Gamma$-liminf inequality.

\begin{prop}\label{prop:lowerbound}
Let $\rho\in L^\infty(\R^n;[0,+\infty))$ be a convolution kernel satisfying {\rm\ref{ass-N1}}-{\rm\ref{ass-N2}}. Assume \eqref{eq: alfabeta}, \eqref{eq:coerciv},  and \eqref{eq: hpsi}. Consider $F_\e$, and $G_\e$ given by \eqref{energies0}, and \eqref{energies1}, respectively.  Let $u\in L^1(\Omega;\R^n)$ and let $A\in\mathcal{A}(\Omega)$, and define $F'(u, A)$ and $G'(u,A)$ by \eqref{eq:Gammaliminf} {in correspondence of $F_\varepsilon$ and $G_\varepsilon$, respectively}. If $F'(u,A)<+\infty$, then $u\in GSBD^p(A)$ and
\begin{equation*}
\displaystyle F'(u,A)\geq  \alpha\int_A W(\mathcal{E}u)\,\mathrm{d}x + \beta\,\int_{J_u\cap A}\phi_\rho(\nu)\,\mathrm{d}\mathcal{H}^{n-1}\,.
\end{equation*}
If it additionally holds $G'(u, A)<+\infty$, then
\begin{equation*}
\displaystyle G'(u,A)\geq  \alpha\int_A W(\mathcal{E}u)\,\mathrm{d}x + \beta\,\int_{J_u\cap A}\phi_\rho(\nu)\,\mathrm{d}\mathcal{H}^{n-1}+\int_A \psi(|u|)\,\mathrm{d}x\,.
\end{equation*}
\end{prop}

\proof
The proof can be obtained by a standard localization method based on Lemma~\ref{lem: lemmasup}. In order to prove, e.g., the second assertion containing an additional term, we can apply Lemma~\ref{lem: lemmasup} to the set function $\mu(A):= G'(u,A)$, which is superadditive on disjoint open sets since $G_\varepsilon(u,\cdot)$ is superadditive as a set function: 
\begin{equation*}
G'(u, A_1\cup A_2)\geq G'(u, A_1)+G'(u, A_2) \quad \mbox{ whenever } A_1,A_2\in\mathcal{A}(\Omega) \mbox{ with } {A}_1\cap {A}_2=\emptyset\,.
\end{equation*}
Then, we consider the positive measure $\lambda(A):= \mathcal{L}^n(A)+\mathcal{H}^{n-1}(J_u\cap A)$ and the sequence $(\varphi_h)_{h\geq0}$ of $\lambda$-measurable functions on $A$ defined as
\begin{equation*}
\varphi_0(x):=
\begin{cases}
\alpha W(\mathcal{E}u(x)) + \psi(|u(x)|)\,, & \mbox{ if }x\in A\backslash J_u\,,\\
0\,, & \mbox{ if }x\in A\cap J_u\,,
\end{cases}
\end{equation*} 
\begin{equation*}
\varphi_h(x):=
\begin{cases}
0\,, & \mbox{ if }x\in A\backslash J_u\,,\\
\beta\phi^{\xi_h}(x)\,, & \mbox{ if }x\in A\cap J_u\,,
\end{cases}
\end{equation*}
where
\begin{equation*}
\phi^{\xi_h}(x)=
\begin{cases}
\tau_{\xi_h}|\langle\nu_u(x),\xi_h\rangle|\,, & \mbox{ if }x\in J_u^{\xi_h}\cap A\,,\\
0\,, & \mbox{ otherwise in }J_u\cap A\,,
\end{cases}
\end{equation*}
for $(\xi_h)_{h\geq1}$ a dense sequence in $\mathbb{S}^{n-1}$.

Now, by virtue of Proposition~\ref{prop:boundsp} it holds that
\begin{equation*}
\mu(A)\geq \int_A\varphi_h\mathrm{d}\lambda
\end{equation*}
for every $h=0,1,\dots$, so that all the assumptions of Lemma~\ref{lem: lemmasup} are satisfied. The assertion then follows once we notice that, taking into account Lemma~\ref{lem:negri1}, it holds 
\begin{equation*}
\sup_{h\geq0}\varphi_h(x)=\varphi(x):=
\begin{cases}
\alpha W(\mathcal{E}u(x)) + \psi(|u(x)|)\,, & \mbox{ if }x\in A\backslash J_u\,,\\
\beta\phi_\rho(\nu_u(x))\,, & \mbox{ if }x\in A\cap J_u\,,
\end{cases}
\end{equation*}
for $\lambda$-a.e. $x\in A$. 
\endproof

\section{Estimate from above of the $\Gamma$\hbox{-}limit}\label{sec:upperbound}

{We denote by $F''$ and $G''$ the upper $\Gamma$\hbox{-}limits of $(F_\varepsilon)$ and $(G_\varepsilon)$, respectively, as defined in \eqref{eq:limsupchar}.} 

%

\begin{prop}\label{prop:upperbound}
Let $u\in GSBD^p(\Omega)\cap L^1(\Omega;\R^n)$. Then
\begin{equation}
F''(u)\leq \displaystyle \alpha \int_\Omega W(\mathcal{E}u)\,\mathrm{d}x +  \beta\,\int_{J_u}\phi_\rho(\nu)\,\mathrm{d}\mathcal{H}^{n-1}\,. 
\label{eq:upperbound}
\end{equation}
If, in addition, it holds that $\int_\Omega \psi(|u|)\,\mathrm{d}x<+\infty$, then 
\begin{equation}
G''(u)\leq \displaystyle \alpha \int_\Omega W(\mathcal{E}u)\,\mathrm{d}x +  \beta\,\int_{J_u}\phi_\rho(\nu)\,\mathrm{d}\mathcal{H}^{n-1} + \int_\Omega \psi(|u|)\,\mathrm{d}x\,.
\label{eq:upperboundG}
\end{equation}
\end{prop}

\proof
We only prove \eqref{eq:upperbound} by using the density result of Theorem~\ref{thm:density}, as \eqref{eq:upperboundG} follows by an analogous construction with the additional property \eqref{eq:densityfidel}. 

In view of Theorem~\ref{thm:density} and remarks below, since we perform a local costruction and by a diagonal argument it is not restrictive to assume that $u\in \mathcal{W}(\Omega;\Rd)$ and that $J_u$ is a closed simplex contained in any of the coordinate hyperplanes, 
that we denote by $K$.

For every $h>0$, let {$K_h:=\cup_{x\in K}S(x,h)$} be the anisotropic $h$-neighborhood of $K$. 
As $K$ is compact and $(n-1)$-rectifiable, it holds (see for instance \cite[Theorem 3.7]{LusVil})
\begin{equation}\label{anisotropicMink}
\lim_{h\to 0}\frac1h\mathcal{L}^n(K_h)=\int_K \phi_\rho(\nu)\,\mathrm{d}\mathcal{H}^{n-1}
\end{equation}
(observe that a factor $2$ is already contained in our definition \eqref{eq:funphi} of $\phi_\rho$). Let $\gamma_\epsilon>0$ be a sequence such that $\gamma_\epsilon/\epsilon\to0$ as $\epsilon\to0$. Notice that, for $\varepsilon$ small,
\begin{equation*}
K \subset K_{\gamma_\varepsilon}\subset\subset K_{\gamma_\varepsilon+\varepsilon}\subset\subset \Omega\,,
\end{equation*} 
recalling that $K \subset \Omega$.
Let $\phi_\epsilon$ be a smooth cut-off function between $ K_{\gamma_\varepsilon}$ and $K_{\gamma_\varepsilon+\varepsilon}$, and set
\begin{equation*}
u_\epsilon(x):=u(x)(1-\phi_\epsilon(x))\,.
\end{equation*}
Since $u\in W^{1,\infty}({\Omega}\backslash J_u;\R^n)$ we have $u_\epsilon\in W^{1,\infty}({\Omega};\R^n)$. Note also that, by the Lebesgue Dominated Convergence Theorem, $u_\epsilon\to u$ in $L^1(\Omega;\mathbb{R}^n)$. Moreover, since $u_\epsilon=u$ on $S(x,\epsilon)\cap\Omega$ if $x\not\in K_{\gamma_\varepsilon+\varepsilon}$, we have
\begin{equation}
F_\epsilon(u_\epsilon)\leq \displaystyle\frac{1}{\varepsilon}\int_\Omega f\left(\varepsilon \int_{S(x,\varepsilon)\cap\Omega}W(\mathcal{E}u(y))\rho_\varepsilon(y-x)\,\mathrm{d}y\right)\,\mathrm{d}x + \beta \, \frac{\mathcal{L}^n(K_{\gamma_\varepsilon+\varepsilon})}{\epsilon}\,.
\label{stimaint}
\end{equation}
Setting
\begin{equation*}
w_\epsilon(x):= \int_{S(x,\varepsilon)\cap\Omega}W(\mathcal{E}u(y))\rho_\varepsilon(y-x)\,\mathrm{d}y\,,
\end{equation*}
we have that $w_\epsilon(x)$ converges to $w(x):=W(\mathcal{E}u(x))$ in $L^1_{\rm loc}(\Omega)$ as $\varepsilon\to0$. Since $f$ complies with \eqref{eq: alfabeta} and it is increasing, there exists $\tilde{\alpha}>\alpha$ such that $f(t)\leq \tilde{\alpha} t$ for every $t\geq0$.
This gives
\begin{equation*}
\frac{1}{\epsilon}f(\epsilon w_\epsilon(x))\leq \tilde{\alpha}w_\epsilon(x) \quad \mbox{ for every }x\in\Omega \mbox{ and every }\epsilon>0\,,
\end{equation*}
and, taking into account that $\displaystyle\lim_{t\to0^+}\frac{f(t)}{t}=\alpha$, we also infer that
\begin{equation*}
\frac{1}{\epsilon}f(\epsilon w_\epsilon(x))\to \alpha w(x) \quad \mbox{ for a.e. }x\in\Omega\,.
\end{equation*}
Thus, by Lebesgue's Dominated Convergence Theorem,
\begin{equation*}
\lim_{\epsilon\to0} \displaystyle\frac{1}{\varepsilon}\int_\Omega f\left(\varepsilon \int_{S(x,\varepsilon)\cap\Omega}W(\mathcal{E}u(y))\rho_\varepsilon(y-x)\,\mathrm{d}y\right)\,\mathrm{d}x = \displaystyle \alpha \int_\Omega W(\mathcal{E}u)\,\mathrm{d}x\,.
\end{equation*}
As $\tfrac{\gamma_\epsilon+\epsilon}{\epsilon}\to1$ as $\epsilon\to0$,
from \eqref{anisotropicMink}, \eqref{stimaint}, the subadditivity of the limsup and \eqref{eq:limsupchar} we get \eqref{eq:upperbound}.
\endproof

\begin{proof}[\it Proof of Theorems \ref{thm:mainresult} and \ref{thm:mainresult2}]
The two results  follow by combining Propositions \ref{prop:compactness}, \ref{prop:lowerbound}, and \ref{prop:upperbound}
\end{proof}

\section*{Acknowledgements}
The authors are members of Gruppo Nazionale per l'Analisi Matematica, la Probabilit\`a e le loro Applicazioni (GNAMPA) of INdAM.

G. Scilla and F. Solombrino have been supported by the Italian Ministry of Education, University and Research through the Project “Variational methods for stationary and evolution problems with singularities and interfaces” (PRIN 2017).

\bibliographystyle{siam}

\bibliography{references}

\end{document}